\documentclass[11pt]{amsart}
\usepackage{amsmath}
\usepackage{color}
\newtheorem{thm}{Theorem}[section]

\newtheorem{lem}[thm]{Lemma}
\newtheorem{prop}[thm]{Proposition}
\theoremstyle{definition}

\theoremstyle{remark}
\newtheorem{rem}[thm]{\bf Remark}

\numberwithin{equation}{section}
\begin{document}
\title[Dense and subspace dense subsets in finite-dimensional spaces]{Dense and subspace dense subsets in finite-dimensional spaces}

\author{Salah Herzi $^{(1)}$ and Habib Marzougui $^{(2)}$}

\address{$^{(1)}$ Salah Herzi, \ University of Carthage, \ Preparatory Engineering Institute of Bizerte, Jarzouna, 7021, Tunisia.}
\email{salahherziamor@gmail.com}
\address{$^{(2)}$ Habib Marzougui, University of Carthage, Faculty
of Science of Bizerte, (UR17ES21), ``Dynamical systems and their applications'', 7021, Jarzouna, Bizerte, Tunisia}
\email{habib.marzougui@fsb.rnu.tn}
\subjclass[2000]{47A16, 47A15 \\
$^{(2)}$ Corresponding author}
\keywords{Finite-dimensional space, dense, subspace}

\begin{abstract}
This note is motivated by the article of Bamerni, Kadets and Kili\c{c}man [J. Math. Anal. Appl. 435 (2), 1812--1815 (2016)]. We consider the remaining problem which claims that if $A$ is a dense subset of a finite dimensional space $X$,
then there is a nontrivial subspace $M$ of $X$ such that $A\cap M$ is dense in $M$. %By a nontrivial subspace $M$ of $X$, we mean that $M$ is non-zero and distinct from $X$.
%First, we give counterexamples to Theorem 2.1 of \cite{bkk} in real and complex dimension 2.
We show that the above problem has a negative answer when $X=\mathbb{K}^{n}$ ($\mathbb{K}= \mathbb{R}$ or $\mathbb{C}$) for every $n\geq 2$. 
\end{abstract}
\maketitle

\section{\bf Introduction }

 In \cite[Theorem 2.1]{bkk}, Bamerni, Kadets and Kili\c{c}man established that if $A$ is a dense subset of a Banach space $X$, then there is a nontrivial closed subspace $M$ of $X$ such that $A\cap M $ is dense in $M$.
 By a nontrivial subspace $M$ of $X$, we mean that $M$ is non-zero and distinct from $X$. We acknowledge here that the authors of \cite{bkk}
   do not clearly mean $X$ of infinite dimension. So Theorem 2.1 in finite dimension remains an open problem.
\medskip

 \textit{Problem 1}. Is Theorem 2.1 of \cite{bkk} true for $X$ of finite dimension?
 \medskip
 
  In the present note, we show that Problem 1 does not hold in finite dimension, that is \cite[Theorem 2.1]{bkk} is not true for every finite dimensional space $X=\mathbb{K}^{n}$, $n\geq 2$, $\mathbb{K} = \mathbb{R} \textrm{ or } \mathbb{C}$. %In the particular case when the dense subset $A$ is an orbit of an abelian semigroup of matrices on $X = \mathbb{K}^{n}$, $n\geq 2$, Problem 1 was completely solved in \cite{hm2} (see comments in Section 3).
 %and can be stated as follows:
 \medskip

In the sequel, % Here and throughout the paper,
$\mathbb{N}$ and $\mathbb{Z}$ denote the sets of non-negative integers and integers,
respectively,  $\mathbb{N}$ while $\mathbb{N}_{0}$ denotes the set of positive integers. For a row vector $v\in\mathbb{K}^{n}$, we will denote its transpose by $v^{T}$. If $v_1,\dots, v_p\in\mathbb{K}^{n}$,  $p\geq 1$, we denote by $\textrm{span}\{v_1,\dots, v_p\}$ the vector subspace of $\mathbb{K}^{n}$ generated by $v_1,\dots, v_p$. A subset $E\subset \mathbb{K}^{n}$ is called \textit{dense} in ${\mathbb{K}}^{n}$ if $\overline{E} =
{\mathbb{K}}^{n}$, where $\overline{E}$
denotes the closure of $E$. 
It is called \textit{nowhere dense} if $\overline{E}$ has empty interior.

 \section{\bf Subspace dense subset of $\mathbb{K}^{n}$}

The aim of this section is to prove the following theorem.
\medskip

\begin{thm}[The case $X=\mathbb{R}^{n}$] \label{t21} Let $n\geq 2$ be an integer and let $\alpha=(\alpha_{1},\dots, \alpha_{n})$ be an $n$-tuple of negative real numbers such that $1, \alpha_{1},\dots, \alpha_{n}$ are linearly independent over $\mathbb{Q}$. Set $A_{\alpha}=\mathbb{N}^{n}+\mathbb{N}[\alpha_{1},\dots, \alpha_{n}]^{T}$. Then $A_{\alpha}$ is dense in $\mathbb{R}^{n}$ and for every nontrivial subspace $M$ of $\mathbb{R}^{n}$, $A_{\alpha}\cap M$ is not dense in $M$.
\end{thm}
\medskip

For this, let recall the following multidimensional real version of Kronecker's
theorem (see e.g. \cite[Lemma 2.2]{sh}%; see also \cite[Theorem 442]{hr}
).
\medskip

\begin{lem}$(${\it Kronecker's Theorem} \cite[Lemma 2.2]{sh}$)$%; \cite{hAaM}, Lemma 4.2
 \label{l22}  Let $n\geq 1$ be an integer and let  $\alpha_{1},\dots,\alpha_{n}$ be negative real numbers such that
	$1, \alpha_{1},\dots,\alpha_{n}$ are linearly independent over $\mathbb{Q}$. Then the set
	 $$\mathbb{N}^{n}+\mathbb{N}[\alpha_{1},\dots,\alpha_{n}]^{T}:=\left\{[s_{1},\dots,s_{n}]^{T} +s[\alpha_{1},\dots,\alpha_{n}]^{T}:
	\ s,s_{1},\dots,s_{n}\in\mathbb{N}\right\}$$ is dense in $\mathbb{R}^{n}$.
\end{lem}

\begin{lem}$($\cite[Lemma 2.1]{sh}$)$ \label{l23}  Let $n, m\geq 1$ be two integers and let $H = \mathbb{Z}u_1 +\dots + \mathbb{Z}u_m$ with $u_k\in \mathbb{R}^{n}$, $k = 1,\dots, m$. If $m\leq n$, then $H$ is nowhere dense in $\mathbb{R}^{n}$.
\end{lem}
\medskip

\textbf{Proof of Theorem \ref{t21}}. First $A_{\alpha}$ is dense in $\mathbb{R}^{n}$ by Lemma \ref{l22}. Now suppose that there is a nontrivial subspace $M$ of $\mathbb{R}^{n}$ of dimension $r$ ($1\leq r\leq n-1$) such that $A_{\alpha}\cap M$ is dense in $M$. One can choose a basis $\{v_{1},v_{2},\dots,v_{r}\}$ of $M$ so that
$$\begin{cases}
v_{1} & = [x_{1,1},x_{2,1},\dots,x_{n,1}]^{T},\\
v_{2} & = [0,x_{2,2},\dots,x_{n,2}]^{T},\\
\vdots\\
v_{r} & = [0,\dots,0,x_{r,r},\dots,x_{n,r}]^{T}
\end{cases}$$  with $x_{i,j}\in \mathbb{R}$.
	\bigskip
	
	$\bullet$ First, we have $x_{1,1}\neq 0$. Indeed, for every $u\in A_{\alpha}\cap M$, there exist $m, m_1,\dots, m_n\in \mathbb{N}$ and $\lambda_{1},\dots,\lambda_{r}\in\mathbb{R}$ such that $$(1) \ \  u =[m_{1},\dots,m_{n}]^{T}+m[\alpha_{1},\dots,\alpha_{n}]^{T}= \lambda_{1}v_{1}+\dots+
		 \lambda_{r}v_{r}.$$ The first coordinate in $(1)$ gives that $m_{1}+m\alpha_{1}=\lambda_{1}x_{1,1}$. If $x_{1,1} = 0$, then $m=0$ and thus  $A_{\alpha}\cap M\subset\mathbb{N}^{n}\cap M$ which is clearly not dense in $M$, a contradiction.\\

		  Set $I=\{i\in\{1,\dots,r\}: x_{i,i}=0\}$ and
$k=\begin{cases}
\min(I) & \textrm{ if } I\neq\emptyset \\
r+1  & \textrm{ if } I=\emptyset
\end{cases}.$
	\medskip
		
Obviously $2\leq k\leq r+1$ and $x_{i,i}\neq 0$, for $1\leq i\leq k-1$. Let us put $$\begin{cases}
v^{(k)}_{1} & = [x_{1,1},x_{2,1},\dots,x_{k,1}]^{T},\\
v^{(k)}_{2} & = [0,x_{2,2},\dots,x_{k,2}]^{T},\\
\vdots\\
 v^{(k)}_{k-1}& = [0,\dots,0,x_{k-1,k-1}, x_{k,k-1}]^{T}.
\end{cases}$$

Then \ \ $v^{(k)}_{1},v^{(k)}_{2},\dots,v^{(k)}_{k-1}$ are linearly independent in $\mathbb{R}^{k}$. Set \\

$\bullet$ $A^{(k)}_{\alpha} = \mathbb{N}^{k}+\mathbb{N}[\alpha_{1},\dots, \alpha_{k}]^{T}$.

$\bullet$ $M^{(k)} = \textrm{span}\{v^{(k)}_{1},v^{(k)}_{2},\dots,v^{(k)}_{k-1}\}.$ 
\medskip

Since $A_{\alpha}\cap M$ is dense in $M$, so $A^{(k)}_{\alpha}\cap M^{(k)}$ is dense in $M^{(k)}$. Then
one can assume, by some elementary transformations of the $v^{(k)}_{i}$, $i=1, \dots, k-1$, that $M^{(k)}$ has a basis  $\{u_{1},u_{2},\dots,u_{k-1}\}$, where
$$\begin{cases}
u_{1} & = [1 ,0,\ldots,0,y_{1}]^{T},\\
u_{2} & = [0 ,1, 0,\ldots,0,y_{2}]^{T},\\
\vdots\\
u_{k-1} & = [0, \ldots,0,1,y_{k-1}]^{T},
\end{cases}$$ with $y_{i}\in\mathbb{R}$, $i=1, \dots, k-1$. Denote by
\medskip

 $\bullet$ $\Omega = \Big\{u=[m_{1},\dots,m_{k}]^{T}+m[\alpha_{1},\dots,\alpha_{k}]^{T}\in  A^{(k)}_{\alpha}\cap M^{(k)}: m\neq 0\Big\}$.

Then $\Omega$ is dense in $M^{(k)}$: indeed, we have  $(A^{(k)}_{\alpha}\cap M^{(k)})\setminus \Omega\subset\mathbb{N}^{k}\cap M^{(k)}$ which is nowhere dense in $M^{(k)}$. Since $A^{(k)}_{\alpha}\cap M^{(k)}$ is dense in $M^{(k)}$, so $\Omega$ is dense in $M^{(k)}$.
\bigskip
	
\textit{In the sequel, let $ u\in \Omega$ be fixed:} 

$$ \  u=[m_{1},\dots,m_{k}]^{T}+m[\alpha_{1},\dots,\alpha_{k}]^{T}
= \lambda_{1}u_{1}+\dots+
\lambda_{k-1}u_{k-1},$$ where $m, m_1,\dots, m_{k}\in \mathbb{N}$, $m\neq 0$ and $\lambda_{1},\dots,\lambda_{k-1}\in\mathbb{R}$.
\medskip	
%Denote by	
%\medskip
%$\bullet$  $\Omega_{u} = \Big\{[t_{1},\dots,t_{k}]^{T}+ t[\alpha_{1},\dots,\alpha_{k}]^{T}\in \Omega: mt_{k-1}-tm_{k-1}\neq 0\Big\}$. \\	
Observe that for every $v\in \Omega$, we have $ v=[t_{1},\dots,t_{k}]^{T}+t[\alpha_{1},\dots,\alpha_{k}]^{T}=
\delta_{1}u_{1}+\dots+\delta_{k-1}u_{k-1},$ where $\delta_{1}, \dots, \delta_{k-1}\in\mathbb{R}$, $t_1, \dots, t_k, t\in \mathbb{N}$ with $t\neq 0$. We obtain for $i=1, \dots, k-1$:\\

$\begin{cases}
t_{i}+t\alpha_{i} & = \delta_{i} \\
m_{i}+m\alpha_{i} & = \lambda_{i} \\
mv-tu & =(mt_{1}-tm_{1})u_{1}+\dots+(mt_{k-1}-tm_{k-1})u_{k-1}
\end{cases}$\\
\medskip
	
We have
\medskip

\qquad $(*) \ \ m_{k}+m\alpha_{k} = \ (m_{1}+m\alpha_{1})y_{1}+
\dots+(m_{k-1}+m\alpha_{k-1})y_{k-1}$. \\

Moreover, the last row of $mv-tu$ gives that\\

\qquad $(**) \ \ mt_{k}-tm_{k} = (mt_{1}-tm_{1})y_{1}+\dots+(mt_{k-1}-tm_{k-1})y_{k-1}$.
\\

 From (**), we distinguish two cases:\\
		
\textit{Case 1: $1, y_{1},\dots,y_{k-1}$ are linearly independent over $\mathbb{Q}$.} \ In this case, for every $v\in  \Omega$, we have from $(**)$, $mt_{i}-tm_{i}=0$ for $i=1,\dots,k$. Therefore $mv=tu$. It follows that $m\Omega\subset\mathbb{N} u$ and thus by density $M^{(k)}=\mathbb{N}u =\mathbb{R}u$. A contradiction.\\
		
\textit{ Case 2: $1, y_{1},\dots,y_{k-1}$ are not linearly independent over $\mathbb{Q}$.} \\ First, observe that  $ y_{1},\dots,y_{k-1}$ are not all in $\mathbb{Q}$; otherwise, $ y_{1},\dots,y_{k-1}\in\mathbb{Q}$ and then by  $(*)$, $m=0$. Therefore $A^{(k)}_{\alpha}\cap M^{(k)}\subset\mathbb{N}^{k}\cap M^{(k)}$ and thus $A^{(k)}_{\alpha}\cap M^{(k)}$ is not dense in $M^{(k)}$. A contradiction. \\
%Assume for example that $y_{1}\not\in\mathbb{Q}$.\\

Second, let $j$ be the maximal number of the rationally independent numbers with $1$ among $y_{1},\dots,y_{k-1}$. This means that there exist \\
$1\leq i_{1}<i_{2}<\dots <i_{j}\leq k-1$ such that:
%; that is  $j=\max(J)$, where $$J=\left\{i\in\{1,\dots,k-1\}: \{1, y_{1},\dots,y_{i}\} \hbox{ is linearly independent over}\  \mathbb{Q}\right\}$$
\medskip

 $-$ $\{1, y_{i_1},\dots,y_{i_j}\}$ is linearly independent over $\mathbb{Q}$,

 $-$ for every $p\in\{1,\dots,k-1\}\setminus\{i_{1},i_{2},\dots,i_{j}\}$,
\\

$(***)\ \   y_{p}=a_{0,p}+a_{1,p}y_{i_1}+\dots+a_{j,p}y_{i_j},$
for some  $a_{0,p},a_{1,p},\dots,a_{j,p}\in\mathbb{Q}$.  
\medskip

Set $K_{j}:=\{1,\dots,k-1\}\setminus\{i_{1},i_{2},\dots,i_{j}\}$. Notice that $1\leq j\leq k-2$.\\
 
% We can assume, by permuting indices if necessary, that $i_{l}= l$, \ $l=1,\dots, j$.	
		
%Then we have, for every $p=j+1,\dots, k-1$, 
%\medskip

%\quad  $(***) \ \  y_{p} = a_{0,p}+a_{1,p}y_{1}+\dots+a_{j,p}y_{j}$, for some $a_{0,p}, a_{1,p},\dots,a_{j,p}\in\mathbb{Q}$.		
%We proceed by induction on $p$. Claim 4 is true for  $p=j+1$, this follows from Claim 3. Now suppose that Claim 4 is true for every $p=j+1,\dots,k-2$; $y_{p} = a_{0,p}+a_{1,p}y_{1}+\dots+a_{j,p}y_{j}$. Then, in particular, for every $i=j+1, \dots, p$, there exist $a_{0,i}, a_{1,i},\dots,a_{j,i}\in\mathbb{Q}$ such that $y_{i} = a_{0,i}+a_{1,i}y_{1}+\dots+a_{j,i}y_{j}$. By Claim 3, there exist $a_{0,p+1}, a_{1,p+1},\dots,a_{p,p+1}\in\mathbb{Q}$ such that $$y_{p+1}=a_{0,p+1}+a_{1,p+1}y_{1}+\dots+a_{p,p+1}y_{p}.$$
%Then we have
%		 \begin{align*}
%		 	y_{p+1} & = a_{0,p+1}+a_{1,p+1}y_{1}+\dots+a_{j,p+1}y_{j}+
%		 	\underset{i=j+1}{\overset{p}{\sum}}a_{i,p+1}y_{i}\\
%		 	& =  a_{0,p+1}+a_{1,p+1}y_{1}+\dots+a_{j,p+1}y_{j}+
%		 	 \underset{i=j+1}{\overset{p}{\sum}}a_{i,p+1}(a_{0,i}+a_{1,i}y_{1}+\dots+a_{j,i}y_{j}).
%		 \end{align*}
%Therefore $$y_{p+1}=b_{0,p+1}+b_{1,p+1}y_{1}+\dots+b_{j,p+1}y_{j},$$ where $b_{0,p+1}, b_{1,p+1},\dots,b_{j,p+1}\in\mathbb{Q}$.
%We conclude that Claim 4 is true for every $p=j+1,\dots, k-1$.
\medskip
	
Now let $v\in \Omega:$ $ v=[t_{1},\dots,t_{k}]^{T}+t[\alpha_{1},\dots,\alpha_{k}]^{T},\ \hbox{with}\ \ t\neq 0$. From $(**)$ and $(***)$, we have that
		\begin{align*}
				mt_{k}-tm_{k}& =  \underset{s=1}{\overset{j}{\sum}}(mt_{i_{s}}-tm_{i_{s}})y_{i_{s}} + 
				\underset{p\in K_{j}}{\overset{}{\sum}}(mt_{p}-tm_{p})
				(a_{0,p}+a_{1,p}y_{i_1}+\dots+a_{j,p}y_{i_j})\\
				& = \underset{p\in K_{j}}{\overset{}{\sum}}(mt_{p}-tm_{p})a_{0,p} +\Big[(mt_{i_1}-tm_{i_1})+\underset{p\in K_{j}}{\overset{}
					{\sum}}(mt_{p}-tm_{p})a_{1,p}\Big]y_{i_1} +
				\\
				& \dots + \Big[(mt_{i_j}-tm_{i_j})+\underset{p\in K_{j}}{\overset{}
					{\sum}}(mt_{p}-tm_{p})a_{j,p}\Big]y_{i_j}.
		\end{align*}
		Since $\{1, y_{i_1},\dots, y_{i_j}\}$ is linearly independent over $\mathbb{Q}$, we obtain that:
		$$mt_{k}-tm_{k} = \underset{p\in K_{j}}{\overset{}{\sum}}(mt_{p}-tm_{p})a_{0,p}$$ and for all  \ $s=1,\dots,j$:
		$$ (mt_{i_s}-tm_{i_s})+\underset{p\in K_{j}}{\overset{}
				{\sum}}(mt_{p}-tm_{p})a_{s,p}=0.$$
			
			Let $(e_1,\dots,e_k)$ be the canonical basis  of $\mathbb{R}^{k}$. Then we have
				\begin{eqnarray*}
					mv-tu &=& \underset{s=1}{\overset{k-1}{\sum}}(mt_{s}-tm_{s})e_{s}+ (mt_{k}-tm_{k})e_{k}\\
					&=&\underset{s=1}{\overset{j}{\sum}}(mt_{i_s}-tm_{i_s})e_{i_s}+\underset{p\in K_{j}}{\overset{}{\sum}}(mt_{p}-tm_{p})e_{p}+(mt_{k}-tm_{k})e_{k}\\
					&=&-\underset{s=1}{\overset{j}{\sum}}(\underset{p\in K_{j}}{\overset{}
						{\sum}}(mt_{p}-tm_{p})a_{s,p})e_{i_s}+
					\underset{p\in K_{j}}{\overset{}{\sum}}(mt_{p}-tm_{p})(e_{p}+a_{0,p}e_{k})\\
					& = & \underset{p\in K_{j}}{\overset{}{\sum}}(mt_{p}-tm_{p})w_{p,j}+
					\underset{p\in K_{j}}{\overset{}{\sum}}(mt_{p}-tm_{p})w_{p,k},
				\end{eqnarray*}
				where $w_{p,j}=-\underset{s=1}{\overset{j}{\sum}}a_{s,p}e_{i_s}$ and 
				$w_{p,k}=e_{p}+a_{0,p}e_{k}$.
				It follows that
				$$mv-tu = t e^{\prime}_{1}+ \underset{p\in K_{j}}{\overset{}{\sum}}t_{p}e^{\prime}_{p},$$
			where $e^{\prime}_{1} = - \underset{p\in K_{j}}{\overset{}{\sum}}m_p(w_{p,j}+w_{p,k})$ and $e^{\prime}_{p} = m(w_{p,j}+w_{p,k})\in \mathbb{R}^{k}$, $p\in K_{j}$.	Hence \quad
				$mv = t (e^{\prime}_{1}+u)+\underset{p\in K_{j}}{\overset{}{\sum}}t_{p}e^{\prime}_{p}.$
				Therefore	
				$$m\Omega\subset \mathbb{N} (e^{\prime}_{1}+u)+\underset{p\in K_{j}}{\overset{}{\sum}}\mathbb{N}e^{\prime}_{p}\subset\textrm{span}\{e^{\prime}_{1}+u, e^{\prime}_{p}: p\in K_{j}\}.$$ 
			Since $ \Omega$ is dense in $M^{(k)}$, so $M^{(k)}\subset \textrm{span}\{e^{\prime}_{1}+u, e^{\prime}_{p}: p\in K_{j}\}$. As dim$M^{(k)}=k-1\leq k-j$, thus $j=1$ and hence $M^{(k)} = \textrm{span}\{e^{\prime}_{1}+u, e^{\prime}_{p}: p\in K_{1}\}$. It follows that $\mathbb{N} (e^{\prime}_{1}+u)+\underset{p\in K_{1}}{\overset{}{\sum}}\mathbb{N}e^{\prime}_{p}\subset M^{(k)}$ is dense in M$^{(k)}$. A contradiction with Lemma \ref{l23}. This completes the proof of Theorem \ref{t21}. \qed
\medskip

\begin{thm}[The case $X=\mathbb{C}^{n}$] \label{t24}
	Let $ n\geq 2$ be an integer and let $\alpha = (\alpha_{1},\dots, \alpha_{n})$, $\beta= (\beta_{1},\dots, \beta_{n})$ be two $n$-tuples of negative real numbers such that $1, \alpha_{1},\dots, \alpha_{n}, \beta_{1},\dots, \beta_{n}$ are linearly independent over $\mathbb{Q}$. Set \\ $A_{\alpha,\beta} = \mathbb{N}^{n}+i\mathbb{N}^{n}+\mathbb{N}[\alpha_{1}+i\beta_{1},\dots, \alpha_{n}+i\beta_{n}]^{T}$. Then $A_{\alpha,\beta}$ is dense in $\mathbb{C}^{n}$ and for every nontrivial subspace $M$ of $\mathbb{C}^{n}$, $A_{\alpha,\beta}\cap M$ is not dense in $M$.
\end{thm}

We use the following complex version of Kronecker's theorem.
\medskip

\begin{lem}[Kronecker's theorem: complex version]\label{t:61} Let $n\in \mathbb{N}_{0}$ and let 
$\alpha_{1},\dots,\alpha_{n}; \beta_{1}, \dots,\beta_{n}$ be negative real numbers such that $1,
	\alpha_{1},\dots,\alpha_{n},\beta_{1},\dots,\beta_{n}$ are linearly independent over $\mathbb{Q}$. Then the set
	 $$\mathbb{N}^{n}+i\mathbb{N}^{n}+\mathbb{N}[\alpha_{1}+i\beta_{1},\dots,\alpha_{n}+i\beta_{n}]^{T}$$ is dense in $\mathbb{C}^{n}$.
\end{lem}

\begin{proof} This results from Lemma \ref{l22} by identifying $\mathbb{C}^{n}$ with $\mathbb{R}^{2n}$ in the natural way.
\end{proof}
\medskip

\begin{proof} [Proof of Theorem \ref{t24}] First $A_{\alpha,\beta}$ is dense in $\mathbb{C}^{n}$ by Lemma \ref{t:61}.
Assume there is a nontrivial subspace $M$ of $\mathbb{C}^{n}$ such that $A_{\alpha,\beta}\cap M$ is dense in $M$. Let $M_{1}, M_{2}$ be two subspaces of  $\mathbb{R}^{n}$ such that $M=M_{1}+ iM_{2}$. By the isomorphism \ $\phi: (x_{1}+iy_{1},\dots,x_{n}+iy_{n})\in\mathbb{C}^{n}\mapsto(x_{1},y_{1};\dots;x_{n},y_{n})
\in\mathbb{R}^{2n},$ we have that $\phi(A_{\alpha,\beta}) = \mathbb{N}^{2n}+\mathbb{N}[\alpha_{1},\beta_{1};\dots; \alpha_{n},\beta_{n}]^{T}$ and \\ $\phi(A_{\alpha,\beta}\cap M) = \phi(A_{\alpha,\beta})\cap \phi(M) = A_{\mu}\cap (M_{1}\times M_{2}),$ where $\mu=[\alpha_{1},\beta_{1};\dots; \alpha_{n},\beta_{n}]^{T}$. Then $A_{\mu}\cap (M_{1}\times M_{2})$ is dense in $M_{1}\times M_{2}$, ($M_{1}\times M_{2}$ is a nontrivial subspace of $\mathbb{R}^{2n}$). This leads to a contradiction with Theorem \ref{t21}.
\end{proof}
\medskip

%\begin{rem} \label{r23} 
We give in the two propositions below other dense subsets $A$ of $\mathbb{K}^{2}$
such that for every straight line $\Delta$ in $\mathbb{K}^{2}$, $A\cap \Delta$ is not dense in $\Delta$.
\medskip

	\begin{prop}[{Counterexample in $\mathbb{C}^{2}$}]
		\label{p21} Let $\theta_{1}, \theta_{2}\in\mathbb{R}$ such that $1, \theta_{1}, \theta_{2}$ are rationally independent and set $A_2=\{[r_{1}e^{2i\pi n_{1}\theta_{1}},r_{2}e^{2i\pi n_{2}\theta_{2}}]^{T}: r_{1}, r_{2}\in \mathbb{R}_{+}^  {*}, n_{1}, n_{2}\in\mathbb{N}\}$. Then $A_2$ is dense in $\mathbb{C}^{2}$, but $A_2\cap \Delta$ is not dense in $\Delta$, for every straight line $\Delta$ in $\mathbb{C}^{2}$.
	\end{prop}
	
	\begin{proof} As $1, \theta_{1}, \theta_{2}$ are rationally independent, then $A$ is dense in $\mathbb{C}^{2}$. We let
		$\Delta = \mathbb{C}u$, where  $u = [a_{1},a_{2}]^{T}\in\mathbb{C}^{2}\setminus \{[0,0]^{T}\}$. Suppose that $A_2\cap \Delta$ is dense in $\Delta$. One can assume that $a_{1}a_{2}\neq 0$ (since $A_2\cap \Delta = \emptyset$ whenever $a_{1}a_{2}=0$).  We let  $\frac{a_{2}}{a_{1}} = \frac{|a_{2}|}{|a_{1}|}e^{2i\pi\theta}$, where $\theta\in\mathbb{R}$.   An element $z\in A_2\cap \Delta$ can be written as $z = [r_{1}e^{2i\pi n_{1}\theta_{1}},r_{2}e^{2i\pi n_{2}\theta_{2}}]^{T} = \lambda [a_{1},a_{2}]^{T}$, where $\lambda\in\mathbb{C}$. We have $\frac{a_{2}}{a_{1}} = \frac{|a_{2}|}{|a_{1}|}e^{2i\pi(n_{2}\theta_{2}-n_{1}\theta_{1})}= \frac{|a_{2}|}{|a_{1}|}e^{2i\pi\theta}$. Hence $\theta = n_{2}\theta_{2}-n_{1}\theta_{1}+ k$ for some $k\in\mathbb{Z}$. The triplet $(n_{1}, n_{2}, k)$ is unique. Indeed, if $(t_{1}, t_{2}, k^{\prime})$ is another triplet such that $\theta=t_{2}\theta_{2}-t_{1}\theta_{1}+ k^{\prime}$, then $(n_{2}-t_{2})\theta_{2}-(n_{1}-t_{1})\theta_{1}+(k- k^{\prime})=0$. As $1, \theta_{1}, \theta_{2}$ are rationally independent, then $n_{2}=t_{2}, n_{1}=t_{1},\ k= k^{\prime}$. Finally, $A\cap \Delta= \mathbb{R}^{*}_{+} v$, where $v = [e^{2i\pi n_{1}\theta_{1}},\frac{|a_{2}|}{|a_{1}|}e^{2i\pi n_{2}\theta_{2}}]^{T}$ and $n_{1}$ and $n_{2}$ are fixed. Clearly $A\cap \Delta$ is not dense in $\Delta$.
	\end{proof}
	
%	The following example is communicated to us by the referee when submitting the first submission (see also Example \ref{e36}).
	\medskip
	
	\begin{prop}[{Counterexample in $\mathbb{R}^{2}$}]
			\label{p12} \ 
			
			Let $B = \left\{r \left(\cos(2\pi \theta, \sin(2\pi \theta)\right): r\in \mathbb{R}_{+}^{*}, \theta\in[0,\frac{1}{2}[\cup \left([\frac{1}{2},1[
		\cap(\mathbb{R}\setminus\mathbb{Q})\right)\right\}$. Then $B$ is dense in $\mathbb{R}^{2}$, but for every straight line through the origin $\Delta$, the intersection $B\cap \Delta$, is not dense in $\Delta$.
	\end{prop}
	
	\begin{proof} $B$ is dense $\mathbb{R}^{2}$ since $[\frac{1}{2},1[
		\cap(\mathbb{R}\setminus\mathbb{Q})$ is dense in $[\frac{1}{2},1[$. Now, $B\cap \Delta$ is only a half line, and thus not dense in $\Delta$.
	\end{proof}	
%\end{rem}
\medskip

\begin{rem}\label{r26} \rm{We point out that when $X$ is of finite dimension, Lemma 2.3 of \cite{bkk} used in the proof of Theorem \cite[Theorem 2.1]{bkk} is not true.
Indeed, let
$X = \mathbb{K}^{n}$, $n\geq 1$. Fix an arbitrary
$(n - 1)$-dimensional subspace $Y$ of $X$, and let
$A = X~\backslash~Y$, which is dense in $X$. Let $e\in X$
with d$(e, Y)>1$, which exists ($d(e, Y) = \inf_{y\in Y}\|e-y\|$, where $\| \ \|$ denotes the euclidean norm in  $\mathbb{K}^{n}$). Then Lemma 2.3 of \cite{bkk} claims that there is some $a\in A$
such that $d(e, \textrm{span}\{Y, a\})
> 1$, which is absurd since $e\in \textrm{span}\{Y, a\}$.}
\end{rem}
\bigskip
 
\textit{Acknowledgements.}
%The authors would like to thank the referee for valuable comments and suggestions.
This work was supported by the research unit: ``Dynamical systems and their applications'' [UR17ES21], Ministry of
Higher Education and Scientific Research, Faculty of Science of Bizerte, Tunisia.

\vskip 0,2 cm
\bibliographystyle{amsplain}

\begin{thebibliography}{9}
%\bibitem{aAhM13}
 % A. Ayadi and H. Marzougui, \emph{Abelian semigroups of matrices on
  %	$\mathbb{C}^{n}$ and hypercyclicity},  Proc. Edinb. Math. Soc. \textbf{57} (2014),  323--338.
%\bibitem{aAhM16} A. Ayadi and H. Marzougui, \emph{Hypercyclic abelian semigroups of matrices on
%$\mathbb{R}^{n}$}, Topology Appl, \textbf{210} (2016),  29--45; corrigendum ibid. 287, Article ID 107330, 8 p.(2021).
\bibitem{bkk} N. Bamerni,  V. Kadets and A. Kili\c{c}man, \emph{ Hypercyclic operators are subspace hypercyclic}, J. Math. Anal. Appl. \textbf{435} (2) (2016), 1812--1815.
%\bibitem{fe1} N.S. Feldman, \emph{Hypercyclic tuples of operators and somewhere dense orbits},
%J. Math. Anal. Appl. \textbf{346} (2008), 82--98 
%\bibitem{hr} G.H. Hardy  and E.M. Wright, An introduction to the Theory of Numbers, Oxford University Press (2008).
%\bibitem{hm} S. Herzi and  H. Marzougui, \emph{Subspace-hypercyclic abelian linear semigroups}, J. Math. Anal. Appl. \textbf{487}, (2020) 123960.
%\bibitem{hm2} S. Herzi and  H. Marzougui, \emph{Subspace-hypercyclic abelian semigroups of matrices on $\mathbb{R}^{n}$}, Results Math., \textbf{76} Article No. 185, (2021). 
%\bibitem{mj}
 %M. Javaheri, \emph{Semigroups of matrices with dense orbits}. Dyn. Syst., \textbf{26} (2011), 235--243.
 \bibitem{sh} S. Shkarin, \emph{Hypercyclic tuples of operator on  ${\mathbb{C}}^{n}$ and ${\mathbb{R}}^{n}$}, Linear Multilinear Algebra, \textbf{60} (2011), 885--896.
\end{thebibliography}
\vskip 0,4 cm
%%% ----------------------------------------------------------------------

\end{document}